\documentclass[a4paper,12pt]{article}
\usepackage{amssymb,amsthm,amsmath,amsfonts}
\usepackage{graphicx}
\usepackage{cite}
\usepackage{color}

\theoremstyle{definition}
\newtheorem{theorem}{Theorem}

\newtheorem{corollary}[theorem]{Corollary}
\newtheorem{proposition}[theorem]{Proposition}
\newtheorem{definition}[theorem]{Definition}

\newtheorem{observation}[theorem]{Remark}

\theoremstyle{remark}
\newtheorem{remark}{Remark}

\begin{document}

\title{A fractional Borel-Pompeiu type formula and a related fractional $\psi-$Fueter operator with respect to a vector-valued function}
\small{
\author
{Jos\'e Oscar Gonz\'alez-Cervantes$^{(1)}$ and Juan Bory-Reyes$^{(2)\footnote{corresponding author}}$}
\vskip 1truecm
\date{\small $^{(1)}$ Departamento de Matem\'aticas, ESFM-Instituto Polit\'ecnico Nacional. 07338, Ciudad M\'exico, M\'exico\\ Email: jogc200678@gmail.com\\$^{(2)}$ {SEPI, ESIME-Zacatenco-Instituto Polit\'ecnico Nacional. 07338, Ciudad M\'exico, M\'exico}\\Email: juanboryreyes@yahoo.com
}}

\maketitle
\begin{abstract} 
In this paper we combine the fractional $\psi-$hyperholomorphic function theory with the fractional calculus with respect to another function.  As a main result, a fractional Borel-Pompeiu type formula related to a fractional $\psi-$Fueter operator with respect to a vector-valued function, is proved. 
\end{abstract}


\section{Introduction} 
The extension of real integer-order derivatives to fractional derivatives is an old topic introduced by Leibnitz in 1695. The date September 30, 1695 is regarded as the exact birth-date of the fractional calculus. For historical review of the theory we refer the reader to \cite{MR, Ro}.

The interest in fractional integrals and derivatives (Fractional Calculus) has been growing continuously during the last few years because of numerous applications in recent studies in engineering science.

There exists a vast literature on different definitions of fractional derivatives, see for instance \cite{MKS,OS,O,OM,P,GM}. In \cite{VOL} the authors introduce a unified two-parametric fractional derivative, from which, all the interesting derivatives can be obtained. The study prevent the ambiguous use of the concept of fractional derivative. For a review of definitions of fractional order derivatives and integrals that appear in mathematics, physics, and engineering we refer the reader to \cite{OT}.

The theory of fractional derivatives and integrals with respect to another function was introduced by Erd\'elyi in \cite{E1,E2}, and it was  extensively studied by Osler in \cite{O1,O2,O3}. This operator theory generalizes several classical fractional derivatives, including Riemann-Liouville and Hadamard derivatives among others, see \cite{B,Ka,S} as well as \cite[Section 2.5]{KST} and \cite[Section 18.2]{SKM}. 
 
The skew field of real quaternions, denoted by $\mathbb H$, in combination with modern analytic methods, give rise to the development of the so-called Quaternionic Analysis. The classical works here are \cite{Fu1,Fu2, sudbery}. Nowadays, it relies heavily on results on functions defined on domains in $\mathbb R^4$ with values in $\mathbb H$, associated to a generalized Cauchy-Riemann operator (the so-called $\psi-$Fueter operator), by using a general orthonormal basis in $\mathbb R^4$ (to be named structural set) $\psi$ of $\mathbb H^4$. The last goes back at least as far as \cite{No}. The theory is centered around the concept of $\psi-$hyperholomorphic functions (i.e., null solutions to the $\psi-$Fueter operator). For direct constructions along classical lines we refer the reader to \cite{MS, S1, shapiro1, shapiro2} and the references given there.

Combining a fractional $\psi-$hyperholomorphic function theory with the fractional calculus with respect to another function, we prove as the main result, a fractional Borel-Pompeiu type formula related to a fractional $\psi-$Fueter operator with respect to a vector-valued function to be introduced here for the first time as far as the authors know. As particular case of this study, the main results of \cite{GB}, are recovered.

The structure of the paper is as follows. After this brief introduction, in the preliminary section we have compiled some basic facts of the  quaternionic analysis associated to a structural set $\psi$, such as Stokes and the Borel-Pompieu formulas related to the $\psi$-Fueter operator, as well as a brief summary of the notions of fractional Riemann-Liouville integral and derivative with respect to another function. Section 3 discusses both Stokes and Borel-Pompieu type formulas related to a fractional $\psi-$Fueter operator with respect to a vector-valued function.

\section{Preliminaries} 
\subsection{Riemann-Liouville fractional integral and derivatives with respect to another function}
Some definitions and standard facts on the fractional calculus with respect to another function are reviewed below.

Let $-\infty <a  < b< \infty$, $\alpha > 0$ and $g\in  C^1([a, b], \mathbb R)$ such that $g'(x)\neq 0$ for all $x \in[a, b]$ a non-negative monotonously increasing and continuous function. 

The (left) Riemann-Liouville fractional integral of order $\alpha$ with respect to $g$ is defined for any $f\in L^1([a, b], \mathbb R)$ by
$$( {}_{a}I_g^{\alpha} f )(x) =
\frac{1}{\Gamma(\alpha)}\int_a^x \frac{  f(y) g' (y) }{(g(x)- g(y))^{1-\alpha}}dy,$$
where $a<x<b$, see \cite{AJ,P}
 
Let $AC^1([a, b], \mathbb R)$ the set of real functions $f$ which are continuously differentiable on $[a, b]$ and $f^{'}$ is absolutely continuous on $[a, b]$. 

The (left) Riemann-Liouville fractional derivative of order $\alpha$ with respect to $g$ for any $f\in AC^1([a, b], \mathbb R)$ is given by

\begin{align*} ({}_aD_g^{\alpha} f) (x) = &
\left(\frac{1}{g'(y) } \frac{d}{dy} \right) ( {}_{a}I_g^{1-\alpha} f )(y) \\
= &\frac{\left(\displaystyle\frac{1}{g'(y)} \displaystyle\frac{d}{dy} \right)}{\Gamma (1-\alpha) } \int_a^x \frac{  f(y) g'(y) }{(g(x)- 
g(y))^{\alpha}}dy.
\end{align*}

The following semigroup property is valid
\begin{align}\label{TFCF}
{}_aD_g^{\alpha} \circ {}_{a}I_g^{\alpha}= I, 
\end{align}
where $I$ is the identity operator, see \cite{AJ}. Note that the fractional Riemann-Liouville integral and derivative with respect to another function are well defined linear operators for any $f\in AC^1([a, b], \mathbb R)$. 

Let us remark that, if $g(x)=x$ (resp. g(x)= ln x), then the Riemann-Liouville fractional integral and derivative with respect to $g$ reduces to the standard Riemann-Liouville (resp. the Hadamard), see \cite{KST, MKS, P}.

\subsection{Rudiments of quaternionic analysis}
The skew field of real quaternions $\mathbb H$ is formed by  $x=x_0+x_{1} {\bf i}+x_{2} {\bf j}+x_{3} {\bf k}$, $x_{k}\in \mathbb R, k= 0,1,2,3$, where   the basic elements satisfy $ {\bf i}^{2}= {\bf j}^{2}= {\bf k}^{2}=-1$, $ {\bf i}\, {\bf j}=- {\bf j}\, {\bf i}= {\bf k};  {\bf j}\, {\bf k}=- {\bf j}\, {\bf k}= {\bf i}$ and $ {\bf k}\, {\bf i}=- {\bf k}\, {\bf i}= {\bf j}$. For $x\in \mathbb H$ we define the mapping of quaternionic conjugation: $x\rightarrow {\overline x}:=x_0-x_{1}{\bf i}-x_{2}{\bf j}-x_{3}{\bf k}$. Easily we see that  $x\,{\overline x}={\overline x}\,x=x^{2}_{0}+x^{2}_{1}+x^{2}_{2}+x^{2}_{3}$ and  $\overline {qx}={\overline x}\,\,{\overline q}$ for $q,x\in \mathbb H$. 

The quaternionic scalar product  of $q, x\in\mathbb H$ is given by 
$$\langle q, x\rangle:=\frac{1}{2}(\bar q x + \bar x q) = \frac{1}{2}(q \bar x + x  \bar q).$$

A set of quaternions $\psi=\{\psi_0, \psi_1,\psi_2,\psi_3\}$ is called structural set if 
$\langle \psi_k, \psi_s\rangle =\delta_{k,s} $, 
for  $k, s=0,1,2,3$ and any quaternion $x$  can be rewritten as $x_{\psi} := \sum_{k=0}^3 x_k\psi_k$, 
where $x_k\in\mathbb R$ for all $k$. Given $q, x\in \mathbb H$ we follow the notation used in {\cite{shapiro1}} to write
$$\langle q, x \rangle_{\psi}=\sum_{k=0}^3 q_k x_k,$$ 
where $q_k, x_k\in \mathbb R$ for all $k$.

Given an structural set $\psi$, we will use the mapping
\begin{equation} \label{mapping}
\sum_{k=0}^3 x_k\psi_k \rightarrow (x_0,x_1,x_2,x_3).
\end{equation}
in essential way.

We have to say something about the set of complex quaternions, which are given by  
$$\mathbb H(\mathbb C) =\{q=q_1+  \textsf{i} \ q_2 \ \mid \ q_1,q_2 \in \mathbb H\},$$
where $\textsf{ i}$   is the imaginary unit of $\mathbb C$.
The main difference to the real quaternions is that not all non-zero elements are invertible. There are so-called zero-divisors.

Let us recall that $\mathbb H$ is embedded in $\mathbb H(\mathbb C)$ as follows:
$$\mathbb H =\{q=q_1+   \textsf{i}  \ q_2 \in \mathbb H(\mathbb C)  \ \mid \  q_1,q_2 \in \mathbb H \ \ \textrm{and} \ \ q_2=0\}.$$
The elements of $\mathbb H$ are written in terms of the structural set $\psi$ hence those of $\mathbb H(\mathbb C)$ can be written as $q=\sum_{k=0 }^3 \psi_k q_k,$ where $q_k\in \mathbb C$.
 
Functions $\mathfrak f$ defined in a bounded domain $\Omega\subset\mathbb H\cong \mathbb R^4$ with value in $\mathbb H$ are considered. They may be written as: $\mathfrak f=\sum_{k=0}^3 f_k \psi_k$, where $f_k, k= 0,1,2,3,$ are $\mathbb R$-valued functions in $\Omega$. Properties as continuity, differentiability, integrability and so on, which as ascribed to $\mathfrak f$ have to be posed by all components $f_k$. We will follow standard notation, for example $C^{1}(\Omega, \mathbb H)$ denotes the set of continuously differentiable $\mathbb H$-valued functions defined in $\Omega$. 
 
The left- and the right-$\psi$-Fueter operators are defined by   
${}^{{\psi}}\mathcal D[\mathfrak f] := \sum_{k=0}^3 \psi_k \partial_k \mathfrak f$ and ${}^{{\psi}}\mathcal  D_r[{ \texttt f }] :=  \sum_{k=0}^3 \partial_k { \texttt f } \psi_k$, for all $\mathfrak{ f}, { \texttt f } \in C^1(\Omega,\mathbb H)$, respectively, where $\partial_k \mathfrak f =\displaystyle \frac{\partial \mathfrak f}{\partial x_k}$ for all $k$, {see \cite{shapiro1, shapiro2}. 

Particularly,  {if $\partial \Omega$ is a 3-dimensional smooth surface then the Borel-Pompieu formula shows that
\begin{align}\label{BorelHyp}  &  \int_{\partial \Omega}(K_{\psi}(y-x)\sigma_{y}^{\psi} \mathfrak{ f}(y)  +  { \texttt f }(y)   \sigma_{y}^{\psi} K_{\psi}(y-x) ) \nonumber  \\ 
&  - 
\int_{\Omega} (K_{\psi} (y-x) {}^{\psi}\mathcal D [\mathfrak{ f}] (y) + {}^{{\psi}}\mathcal  D_r [{ \texttt f }] (y) K_{\psi} (y-x)
     )dy   \nonumber \\
		=  &  \left\{ \begin{array}{ll}  \mathfrak{ f}(x) + { \texttt f }(x) , &  x\in \Omega,  \\ 0 , &  x\in \mathbb H\setminus\overline{\Omega}.                     
\end{array} \right. 
\end{align} 
{Differential and integral versions of Stokes' formulas for the $\psi$-hyperholomorphic functions theory are given by 
\begin{align}\label{StokesHyp} \int_{\partial \Omega} { \texttt f }\sigma^\psi_x \mathfrak{ f} =  &   \int_{\Omega } \left( { \texttt f } 
{}^\psi \mathcal  D[\mathfrak {f}] + {}^{{\psi}}\mathcal  D_r[{ \texttt f }] \mathfrak {f}\right)dx,
\end{align}
for all $\mathfrak {f},{ \texttt f } \in C^1(\overline{\Omega}, \mathbb H)$},  {see \cite{shapiro1, shapiro2, sudbery}}. Here, $d$ stands for the exterior differentiation operator, $dx$ denotes the differential form of the 4-dimensional volume in $\mathbb R^4$ and  
$$\sigma^{{\psi} }_{x}:=-sgn\psi \left( \sum_{k=0}^3 (-1)^k \psi_k d\hat{x}_k\right)$$ 
is the quaternionic differential form of the 3-dimensional volume in $\mathbb R^4$ according  to $\psi$, where $d\hat{x}_k  = dx_0 \wedge dx_1\wedge dx_2  \wedge  dx_3 $ omitting factor $dx_k$.  {In addition,} $sgn\psi$ is $1$, or $-1$,  if  $\psi$ and  $\psi_{std}:=\{1, {\bf i}, {\bf j}, {\bf k}\}$ have the same orientation, or not, respectively.  {Note that}, $|\sigma^{{\psi} }_{x}| = dS_3$  is the differential form of the 3-dimensional volume in $\mathbb R^4$ and write $\sigma_x=\sigma^{{\psi_{std}} }_{x}$. Let us recall that the $\psi$-hyperholomorphic Cauchy Kernel is given by 
 \[ K_{\psi}(y- x)=\frac{1}{2\pi^2} \frac{ \overline{y_{\psi} - x_{\psi}}}{|y_{\psi} - x_{\psi}|^4},\]
and the integral operator 
$${}^{\psi}\mathcal T[\mathfrak{ f}](x) = \int_{\Omega} K_{\psi} (y-x) \mathfrak{ f}  (y) dy$$ 
defined for all $\mathfrak {f}\in L_2(\Omega,\mathbb H)\cup C(\Omega,\mathbb H)$ satisfies 
\begin{align}\label{FueterInv}{}^{\psi}\mathcal D \circ{}^{\psi}\mathcal T[\mathfrak{ f}]=\mathfrak {f}, \ \ \forall \mathfrak{ f}\in L_2(\Omega,\mathbb H)\cup C(\Omega,\mathbb H).
\end{align} 
This can be found in \cite{MS, S1, shapiro1, shapiro2}.  

\section{Main results}
For simplicity of notation, we write $\vec{\alpha}$ and $\vec{\beta}$ instead of the vectors $(\alpha_0, \alpha_1,\alpha_2,\alpha_3)$ and $(\beta_0, \beta_1, \beta_2, \beta_3)$ both in $(0,1)^4$.

\subsection{Fractional $\psi$-Fueter operator of order $\vec{\alpha}$}
\begin{definition}
Let $a=\sum_{k=0}^3\psi_k a_k, b=\sum_{k=0}^3\psi_k b_k \in \mathbb H$ such that $a_k< b_k$ for all $k$. Write 
\begin{align*}  {J_a^b }:= &  \{  \sum_{k=0}^3\psi_k x_k \in \mathbb H \ \mid \ a_k< x_k < b_k, \  \ k=0,1,2,3\} \\
 = & (a_0,b_0) \times (a_1,b_1) \times (a_2,b_2)  \times (a_3,b_3) ,
\end{align*}
and define $m(J_a^b):=(b_0-a_0) (b_1-a_1)(b_2-a_2)(b_3-a_3)$. 
\end{definition}

Set $\mathfrak{ f}=\sum_{i=0}^3\psi_i f_i\in AC^1(J_a^b,\mathbb H)$; i.e., the real components $f_i, i=0,1,2,3$ of $\mathfrak{f}$, belongs to $AC^1((a_i,b_i),\mathbb R)$. 

The mapping $x_j \mapsto f_i(q_0,\dots,x_j,\dots, q_3)$ belongs to $AC^1((a_i, b_i), \mathbb R)$ for each $q\in J_a^b$ and all $i, j=0,1,2,3$.

Now, given $q, x\in  J_a^b $ and $i, j=0,\dots, 3$, the (left) fractional Riemann-Liouville integral of order {$\alpha_j$} with respect to a monotonously increasing functions $g_j\in  C^1([a_j, b_j],\mathbb R)$ for the mapping $x_j \mapsto f_i(q_0,\dots,x_j,\dots, q_3 )$ is defined by 
$$   ({\bf I}_{a_j^+, g _j}^{      {  \alpha_j  }   } f_i)(q_0, \dots, x_j, \dots, q_3)  
=\frac{1}{\Gamma({  \alpha_j  }  )} \int_{a_j}^{x_j} \frac{f_i  
(q_0, \dots,  y _j, \dots, q_3) g _j'( y _j)
}{(g _j(x_j)-  g _j( y _j) )^{1-   {  \alpha_j }  }  } d y _j.$$ 
By the above, as $\displaystyle \mathfrak{ f}=\sum_{i=0}^3 \psi_i f_i$ it follows that 
\begin{align*}  ({\bf I}_{a_j^+,g_j}^{    {  \alpha_{j}  } } \mathfrak{ f} )(q_0, \dots, x_j, \dots, q_3) := &
\frac{1}{\Gamma(    {  \alpha_j }   )} \int_{a_j}^{x_j} \frac{\mathfrak{ f}  
(q_0, \dots,  y _j, \dots, q_3) g _j'( y _j)
}{ (g _j(x_j)-  g _j( y _j) )^{1-   {  \alpha_j }  } }  d y _j  \\
= & 
\sum_{i=0}^3 \psi_i ({\bf I}_{{a_j,g_j}^+}^{ { \alpha_j }   } f_i)(q_0, \dots, x_j, \dots, q_3)  .
\end{align*}     
for every $\mathfrak{ f} \in AC^1(J_a^b,\mathbb H)$ and $q, x \in J_a^b$.

What is more, the (left) fractional Riemann-Liouville derivative of order $\alpha_j$ with respect to a monotonously increasing function $g_k\in  C^1[a, b]$ for all $k=0,1,2,3$ with $g_k' \neq 0$ for $k = 0,1,2,3$ for the mapping $x_j\mapsto \mathfrak{ f}(q_0, \dots, x_j, \dots, q_3)$is given as
\begin{align*} 
D _{a_j^+, g_j }^{      {  \alpha_j  }   } \mathfrak{ f} (q_0, \dots ,  x_j ,\dots ,  q_3) = &
\left(\frac{1}{g_j'(x_j) }  \frac{\partial   }{\partial x_j} \right) 
 \sum_{i=0}^3 \psi_i ({\bf I}_{a_j^+, g_j}^{      {  \alpha_j  }    } f_i)(q_0, \dots, x_j, \dots, q_3) \\
 = &
 \frac{\left(\displaystyle\frac{1}{g_j'(x_j) } \displaystyle \frac{\partial   }{\partial x_j} \right)}{\Gamma({\alpha_j})} \int_{a_j}^{x_j} \frac{ \mathfrak{f} (q_0, \dots,  y _j, \dots, q_3) g_j'( y _j)}{(g_j (x_j)-  g_j( y _j) )^{1-{\alpha_j}}}dy_j .
\end{align*}
\begin{remark}
Note that ${\bf I}_{a_j^+, g_j}^{{\alpha_j}}\mathfrak{ f}$ and $D _{a_j^+,g_j}^{{\alpha_j}}\mathfrak{ f}$ are $\mathbb H(\mathbb C)$-valued functions for every $j$. In a similar way we can introduce the (right) fractional Riemann-Liouville integral and derivative, to be denoted by $({\bf I}_{b_j^-, g_j}^{{\alpha_j}}\mathfrak{ f})$ and $D _{b_j^-, g_j}^{{\alpha_j}}\mathfrak{ f}$ respectively, but we will not develop this point here. 
\end{remark}
\begin{definition} Let $\mathfrak{ f}, { \texttt f } \in AC^1(J_a^b,\mathbb H)$ and let the vector-valued function ${\bf g}:=(g_0,g_1,g_2,g_3)$ with  monotonously increasing components $ g_k\in  C^1[a_k, b_k]$  with $g_k' \neq 0$ for $k = 0,1,2,3$. The (left) fractional $\psi$-Fueter operator of order $\vec{\alpha}$  with respect to ${\bf g}$ is defined by
\begin{align*} 
{}^{\psi}\mathfrak D_{a,{\bf g}}^{\vec{\alpha}}[\mathfrak{ f}] (q,x):= & \sum_{j=0}^3 \psi_j( D _{a_j^+, g_j}^{{\alpha_j}}\mathfrak{ f})(q_0, \dots, x_j , \dots,  q_3)    \\
= &
\sum_{j=0}^3 \psi_j  \frac{\left(\displaystyle\frac{1}{g_j'(x_j) }  \displaystyle\frac{\partial   }{\partial x_j} \right)}{\Gamma(      {  \alpha_j  }   )} \int_{a_j}^{x_j} \frac{\mathfrak{ f}  
(q_0, \dots,  y _j, \dots, q_3) g_j'( y _j)
}{(g_j (x_j)-  g_j( y _j) )^{1-      {  \alpha_j  }      }} d y _j.
\end{align*}
Particularly, ${}^{\psi}\mathfrak D_{a,{\bf g} }^{\vec{\alpha}}[\mathfrak{f}] (q,x) \mid_{x=q} ={}^{\psi}\mathfrak D_{a,{\bf g} }^{\vec{\alpha}}[\mathfrak{f}] (q) $; i.e., it is   ${}^{\psi}\mathfrak D_{a,{\bf g}}^{\vec{\alpha}}[\mathfrak{f}] $ at point $q$.
\end{definition}
Observe that  $q$ is considered a fixed point since the integration and derivation variables are the real components of $x$ and ${}^{\psi}\mathfrak D_{a,{\bf g} }^{\vec{\alpha}}[{ \mathfrak f }](q,\cdot)$ is a $\mathbb H(\mathbb C)$-valued function.

\begin{remark}
Taking into account the non-commutativity of the $\mathbb H-$multiplication it is natural to introduce the right hand side analogue of ${}^{\psi}\mathfrak D_{a,{\bf g}}^{\vec{\alpha}}[\mathfrak{ f}]$:
\begin{align*} 
 {}^{\psi}\mathfrak D_{r,a,{\bf g}}^{\vec{\alpha}}[{ \texttt f }] (q,x):= & 
\sum_{j=0}^3  \frac{\left(\displaystyle\frac{1}{g_j'(x_j) } \displaystyle \frac{\partial   }{\partial x_j} \right)}{\Gamma(      {  \alpha_j  }   )} \int_{a_j}^{x_j} \frac{{ \texttt f }  (q_0, \dots,  y _j, \dots, q_3) g_j'( y _j)
}{(g_j (x_j)-  g_j( y _j) )^{1-      {  \alpha_j  }      }} d y _j  \psi_j,
\end{align*}
for $q, x\in J_a^b$. 
\end{remark}

A key observation is that if $g_k(x)= x$ for all $x\in [a_k, b_k]$ and  $k=0,1,2,3$ then ${}^{\psi}\mathfrak D_{a,{\bf g}}^{\vec{\alpha}}$ becomes at the quaternionic fractional operator presented in \cite{GB}. Therefore,  for  $0<a_k<b_k$ with   $k=01,2,3$ considering ${\bf ln}(x) = (\ln x_0,\ln x_1, \ln x_2, \ln x_3)$  for all $x_k\in [a_k, b_k]$ and  $k=0,1,2,3$ we obtain the fractional $\psi$-Fueter operator of order $\vec{\alpha}$  with respect to ${\bf ln}$ associated to the Hadamard fractional derivative.  
\begin{definition}
Given $\mathfrak{f}, { \texttt f } \in AC^1(J_a^b,\mathbb H)$   define 
\begin{align*}    
   {}^{\psi}\mathcal I_{a,{\bf g} }^x [\mathfrak{f} ] (q,x,\vec{\alpha})    
	 := &\sum_{j=0}^3 \frac{1}{\Gamma(      {  \alpha_j  }   )} \int_{a_j}^{x_j} \frac{\mathfrak{f}   
(q_0, \dots,  y _j, \dots, q_3) g_j'( y _j)
}{(g_j (x_j)-  g_j( y _j) )^{1-      {  \alpha_j  }      }} d y _j \\
	= &\sum_{j=0}^3  ({\bf I}_{a_j^+,g_j}^{    { 1- \alpha_{j}  } } {\mathfrak f} )(q_0, \dots, x_j, \dots, q_3) ,\\
	 {}^{\psi}\mathfrak C_{a,{\bf g}}^{\vec{\alpha}}[{\mathfrak f}](q, x):= &   \sum_{j=0}^3 (g_j'(x_j)-1)   \psi_j \dfrac{1}{g_j'(x_j)} \dfrac{\partial}{\partial x_j} ({\bf I}_{a_j^+,g_j}^{    {  \alpha_{j}  } } {\mathfrak f} )(q_0, \dots, x_j, \dots, q_3) \\
=&  \sum_{j=0}^3 (g_j'(x_j) - 1 ) \psi_j D _{a_j^+, g_j }^{      {  \alpha_j  }   } [{\mathfrak f} ](q_0, \dots ,  x_j ,\dots ,  q_3) ,\\
{}^{\psi}\mathfrak C_{r,a,{\bf g} }^{\vec{\alpha}}[{{ \texttt f }}  ](q, x)  := &    \sum_{j=0}^3 (g_j'(x_j) - 1 )  D_{a_j^+, g_j }^{      {  \alpha_j  }   } [{{ \texttt f }} ](q_0, \dots ,  x_j ,\dots ,  q_3)  \psi_j,\\
{}^{\psi}\mathfrak P_{a,{\bf g} }^{\vec{\alpha}}[{\mathfrak f} ](q, x)  := &  \sum_{j=0}^3  D _{a_j^+, g_j }^{ 1- \alpha_j } [{\mathfrak f} ](q_0, \dots ,  x_j ,\dots ,  q_3)  ,
\end{align*}
\end{definition}
	
\begin{observation}
If $0<a_k<b_k $ for $k=0,1,2,3$ the previous operators, associated to the Hadamard fractional derivative, are given by
 \begin{align*} 
{}^{\psi}\mathfrak D_{a,{\bf ln} }^{\vec{\alpha}}[{\mathfrak f} ] (q,x) =  & \sum_{j=0}^3 \psi_j  \frac{ x_j }{\Gamma(      {  \alpha_j  }   )} \frac{\partial   }{\partial x_j} \int_{a_j}^{x_j} \frac{ {\mathfrak f}   
(q_0, \dots,  y _j, \dots, q_3)  
}{  y _j  (\ln(x_j)-  \ln( y _j) )^{1-      {  \alpha_j  }      }} d y _j , \\
 {}^{\psi}\mathfrak D_{r,a,{\bf ln} }^{\vec{\alpha}}[{ \texttt f } ] (q,x) = & 
\sum_{j=0}^3  \frac{ x_j   }{\Gamma(        \alpha_j     ) }  \frac{\partial   }{\partial x_j} \int_{a_j}^{x_j} \frac{{ \texttt f }  (q_0, \dots,  y _j, \dots, q_3)   
}{   y _j (\ln(x_j)-  \ln( y _j) )^{1-      \alpha_j        }} d y _j  \psi_j ,\\
   {}^{\psi}\mathcal I_{a,{\bf ln } }^x [\mathfrak {f}] (q,x,\vec{\alpha})    
	 := &\sum_{j=0}^3 \frac{1}{\Gamma(      {  \alpha_j  }   )} \int_{a_j}^{x_j} \frac{\mathfrak {f}  
(q_0, \dots,  y _j, \dots, q_3)   
}{ y _j (\ln _j (x_j)-  \ln _j( y _j) )^{1-      {  \alpha_j  }      }  } d y _j ,\\
	 {}^{\psi}\mathfrak C_{a,{\bf ln} }^{\vec{\alpha}}[\mathfrak {f}](q, x):= &   \sum_{j=0}^3 (  1- x_j )   \psi_j 
	\dfrac{\partial}{\partial x_j} ({\bf I}_{a_j^+,\ln _j}^{    {  \alpha_{j}  } } \mathfrak {f} )(q_0, \dots, x_j, \dots, q_3) .
\end{align*}		
for any $\mathfrak {f},  { \texttt f } \in AC^1(J_a^b,\mathbb H)$. 		
\end{observation}
		
\begin{proposition} \label{identities}
Assume that $\mathfrak{f}, {\texttt{f}} \in AC^1(J_a^b,\mathbb H)$. Then we have  
\begin{enumerate}
\item \begin{align*}
{}^{\psi}\mathcal D_x \circ    {}^{\psi}\mathcal I_{a,{\bf g}}^x [\mathfrak{f}] (q,x,\vec{\alpha}) = &  {}^{\psi}\mathfrak C_{a,{\bf g} }^{\vec{\alpha}}[\mathfrak{f}](q, x) + {}^{\psi}\mathfrak D_{a,{\bf g} }^{\vec{\alpha}}[\mathfrak{f}](q, x)  ,
\\
 {}^{\psi}\mathcal D_{r,x} \circ    {}^{\psi}\mathcal I_{a,{\bf g} }^x [{ \texttt f } ] (q,x,\vec{\alpha}) = &  {}^{\psi}\mathfrak C_{r,a,{\bf g}}^{\vec{\alpha}}[{ \texttt f }](q, x) + {}^{\psi}\mathfrak D_{r,a,{\bf g} }^{\vec{\alpha}}[{ \texttt f }](q, x),
\end{align*}
\item 
\begin{align*}
 {}^{\psi}\mathfrak P_{a,{\bf g} }^{\vec{\alpha}}\circ {}^{\psi}\mathcal I_{a,{\bf g} }^x [\mathfrak{f}] (q,x,\vec{\alpha}) = & 
\sum_{j= 0}^3   \mathfrak{f}  (q_0, \dots, x_k , \dots, q_3)  \\ 
 &  \ +  \sum_{{   { \begin{array}{c}j, k=0 \\ j\neq k  \end{array}} }}^3 
 \frac{( {\bf I}_{a_k ^+,g_k }^{    { 1- \alpha_{k }  } } \mathfrak{f} )(q_0, \dots, x_k , \dots, q_3)}{ \Gamma(1-\alpha_j)   (   g_j(x_j) - g_j(a) )^{1-\alpha_j}  }   , 
\end{align*}
\item 
\begin{align*}
{}^{\bar \psi}\mathcal D_x  \circ {}^{\psi}\mathfrak D_{a,{\bf g} }^{\vec{\alpha}}[\mathfrak{f}](q, x)  = &    \Delta_{\mathbb R^2} \circ    {}^{\psi}\mathcal I_{a,{\bf g} }^x [\mathfrak{f} ] (q,x,\vec{\alpha}) - {}^{\bar \psi}\mathcal D_x \circ {}^{\psi}\mathfrak C_{a,{\bf g} }^{\vec{\alpha}}[\mathfrak{f}](q, x)    ,
 \\
 {}^{\bar\psi}\mathcal D_{r,x} \circ{}^{\psi}\mathfrak D_{r,a,{\bf g} }^{\vec{\alpha}}[{ \texttt f }](q, x) = &   \Delta_{\mathbb R^2} \circ    {}^{\psi}\mathcal I_{a,{\bf g} }^x [{ \texttt f }] (q,x,\vec{\alpha})  - {}^{\bar\psi}\mathcal D_{r,x} \circ  {}^{\psi}\mathfrak C_{r,a,{\bf g} }^{\vec{\alpha}}[{ \texttt f }](q, x)  ,\\
	  \end{align*}
\end{enumerate}
		\end{proposition}
	\begin{proof}
\begin{enumerate}
\item and 3. Follow from direct computations.
\item 	From \eqref{TFCF}
\begin{align*}  & {}^{\psi}\mathfrak P_{a,{\bf g} }^{\vec{\alpha}}\circ {}^{\psi}\mathcal I_{a,{\bf g} }^x [\mathfrak{f}] (q,x,\vec{\alpha}) =  
\sum_{j=k=0}^3  D _{a_j^+, g_j }^{ 1- \alpha_j } ( {\bf I}_{a_k ^+,g_k }^{    { 1- \alpha_{k }  } } \mathfrak{f} )(q_0, \dots, x_k , \dots, q_3)  \\ 
= &  
\sum_{j= 0}^3   \mathfrak{f}  (q_0, \dots, x_k , \dots, q_3)  +  \sum_{{   { \begin{array}{c}j, k=0 \\ j\neq k  \end{array}} }}^3 
 \frac{( {\bf I}_{a_k ^+,g_k }^{    { 1- \alpha_{k }  } } \mathfrak{f} )(q_0, \dots, x_k , \dots, q_3)}{ \Gamma(1-\alpha_j)   (   g_j(x_j) - g_j(a) )^{1-\alpha_j}  }   ,\\ 
 \end{align*}
	where the identity
	\begin{align*}  D _{a_j^+, g_j }^{ 1- \alpha_j }  [1] = & \frac{1}{\Gamma(1-\alpha_j)} \frac{1}{g'_j(x_j)}   \frac{\partial}{ \partial x_j}  \int_{a_j}^{x_j} \frac{g'_j (y_j)  }{ ( g_j (x_j) - g_j (y_j) )^{1-\alpha_j}    } dy_j\\
		= &\frac{1}{ \Gamma(1-\alpha_j)   (   g_j(x_j) - g_j(a) )^{1-\alpha_j}  }, 
		\end{align*}
was applied.
	\end{enumerate}
	\end{proof}
\begin{proposition}\label{Stokes}(Stokes type integral formula induced by  ${}^{\psi}\mathfrak D_{a,{\bf g} }^{\vec{\alpha}}$)  
Suppose that ${\bf g}= (g_0,g_1,g_2,g_3)$ and ${\bf h}= (h_0,h_1,h_2,h_3)$ are two vector-valued functions with monotonously increasing components such that $g_k, h_k\in  C^1[a_k, b_k]$ and $g_k' \neq 0 \neq h_k'$ for  $k = 0,1,2,3$.  If $\mathfrak{f},{ \texttt f } \in AC^1(\overline{J_a^b}, \mathbb H)$   consider $q\in J_a^b$ such that   
 the mappings $x\mapsto {}^{\psi}\mathcal I_a^x [\mathfrak{f}](q,x, \vec{\alpha})$ and $ x\mapsto {}^{\psi} \mathcal I_a^x [{ \texttt f }](q,x, \vec{\beta} )$ belong to  $ C^1(\overline{J_a^b}, \mathbb H(\mathbb C))$. 
 Then 
{\begin{align*} &   \int_{\partial J_a^b} {}^{\psi}\mathcal I_{a,{\bf h} }^x  [ { \texttt f }](q, x, \vec{\beta}) \sigma^{{\psi} }_x 
{}^{\psi}\mathcal I_{a,{\bf g} }^x [\mathfrak{f}] (q,x,\vec{\alpha})\\ 
  = &       \int_{J_a^b }  \left( {}^{\psi} \mathcal I_{a,{\bf h} }^x [ { \texttt f }](q, x,\vec{\beta}) \ {}^{\psi}\mathfrak C_{a,{\bf g} }^{\vec{\alpha}}[\mathfrak{f}](q, x) + \  {}^{\psi}\mathfrak C_{r,a,{\bf h} }^{\vec{\beta}}[{ \texttt f }](q, x) {}^{\psi}\mathcal I_{a,{\bf g}}^x [\mathfrak{f}](q, x,\vec{\alpha})\right)dx \\ 
  & +    
   \int_{J_a^b }  \left( {}^{\psi} \mathcal I_{a,{\bf h} }^x [{ \texttt f }](q, x,\vec{\beta}) \ {}^{\psi}\mathfrak D_{a,{\bf g} }^{\vec{\alpha}}[\mathfrak{f}](q, x) + \  {}^{\psi}\mathfrak D_{r,a,{\bf h} }^{\vec{\beta}}[{ \texttt f }](q, x) {}^{\psi}\mathcal I_{a,{\bf g}}^x [\mathfrak{f}](q, x,\vec{\alpha})\right)dx .
\end{align*} }
 \end{proposition}
\begin{proof}
Considering in \eqref{StokesHyp} the functions ${}^{\psi}\mathcal I_{a,{\bf h} }^x [ { \texttt f }](q, x, \vec{\beta})$ and ${}^{\psi}\mathcal I_{a,{\bf g}}^x [\mathfrak{f}] (q,x,\vec{\alpha})$ and use Proposition \ref{identities}.
\end{proof}

\begin{corollary}(A version of the Cauchy theorem)  
Under the hypothesis of Proposition \ref{Stokes}, if moreover 
{$${}^{\psi}\mathfrak D_{a,{\bf g} }^{\vec{\alpha}}[\mathfrak{f}](q, \cdot ) =  {}^{\psi}\mathfrak D_{r,a,{\bf h} }^{\vec{\beta}}[{ \texttt f }](q, \cdot) =0, \quad \textrm{on} \quad J_a^b.$$ }
Then  
 \begin{align*} &   \int_{\partial J_a^b} {}^{\psi}\mathcal I_{a,{\bf h} }^x  [{ \texttt f }](q, x, \vec{\beta}) \sigma^{{\psi} }_x 
{}^{\psi}\mathcal I_{a,{\bf g} }^x [\mathfrak{f}] (q,x,\vec{\alpha}) \\
= &   \int_{J_a^b }  ( {}^{\psi} \mathcal I_{a,{\bf h} }^x [{ \texttt f }](q, x,\vec{\beta}) \ {}^{\psi}\mathfrak C_{a,{\bf g} }^{\vec{\alpha}}[\mathfrak{f}](q, x)  + \  {}^{\psi}\mathfrak C_{r,a,{\bf h}}^{\vec{\beta}}[{ \texttt f }](q, x) {}^{\psi}\mathcal I_{a,{\bf g} }^x [\mathfrak{f}](q, x,\vec{\alpha}) ) dx .
\end{align*}
\end{corollary}

\begin{observation} 
The Stokes type integral formula induced by ${}^{\psi}\mathfrak D_{a,{{\bf ln}}}^{\vec{\alpha}}$ associated to Hadamard fractional derivative holds. 
Set $0<a_k<b_k$ for $k=0,1,2,3$ and let $\mathfrak{f},\texttt{f} \in AC^1(\overline{J_a^b}, \mathbb H)$. Consider $q\in J_a^b$ such that   
 the mappings $x\mapsto {}^{\psi}\mathcal I_a^x [\mathfrak{f}](q,x, \vec{\alpha})$ and $ x\mapsto {}^{\psi} \mathcal I_a^x [\texttt{f}](q,x, \vec{\beta} )$ belong to  $ C^1(\overline{J_a^b}, \mathbb H(\mathbb C))$. Then 
 \begin{align*} &   \int_{\partial J_a^b} {}^{\psi}\mathcal I_{a,{\bf ln} }^x  [\texttt{f}](q, x, \vec{\beta}) \sigma^{{\psi} }_x 
{}^{\psi}\mathcal I_{a,{\bf ln} }^x [\mathfrak{f}] (q,x,\vec{\alpha})\\ 
 & -     \int_{J_a^b }  \left( {}^{\psi} \mathcal I_{a,{\bf ln} }^x [\texttt{f}](q, x,\vec{\beta}) 
\ {}^{\psi}\mathfrak C_{a,{\bf ln} }^{\vec{\alpha}}[\mathfrak{f}](q, x) + \  {}^{\psi}\mathfrak C_{r,a,{\bf ln} }^{\vec{\beta}}[\texttt{f}](q, x) {}^{\psi}\mathcal I_{a,{\bf ln} }^x [\mathfrak{f}](q, x,\vec{\alpha})\right)dx \\ 
= &   
   \int_{J_a^b }  \left( {}^{\psi} \mathcal I_{a,{\bf ln}  }^x [\texttt{f}](q, x,\vec{\beta}) \ {}^{\psi}\mathfrak D_{a,{\bf ln} }^{\vec{\alpha}}[\mathfrak{f}](q, x) + \  {}^{\psi}\mathfrak D_{r,a,{\bf ln} }^{\vec{\beta}}[\texttt{f}](q, x) {}^{\psi}\mathcal I_{a,{\bf ln} 
	}^x [\mathfrak{f}](q, x,\vec{\alpha})\right)dx 
\end{align*}
and if 
{ $${}^{\psi}\mathfrak D_{a,{\bf ln} }^{\vec{\alpha}}[\mathfrak{f}](q,\cdot ) =  {}^{\psi}\mathfrak D_{r,a,{\bf ln} }^{\vec{\beta}}[\texttt{f}](q, \cdot) =0, \quad \textrm{on} \quad J_a^b,$$ }
then  
 \begin{align*}    \int_{\partial J_a^b} {}^{\psi}\mathcal I_{a,{\bf ln} }^x  [\texttt{f}](q, x, \vec{\beta}) \sigma^{{\psi} }_x 
{}^{\psi}\mathcal I_{a,{\bf ln} }^x [\mathfrak{f}] (q,x,\vec{\alpha})= &   \int_{J_a^b }  ( {}^{\psi} \mathcal I_{a,{\bf ln} }^x [\texttt{f}](q, x,\vec{\beta}) \ {}^{\psi}\mathfrak C_{a,{\bf ln} }^{\vec{\alpha}}[\texttt{f}](q, x)  \\
 &  + \  {}^{\psi}\mathfrak C_{r,a,{\bf ln} }^{\vec{\beta}}[\texttt{f}](q, x) {}^{\psi}\mathcal I_{a,{\bf ln} }^x [\mathfrak{f}](q, x,\vec{\alpha}) )dx .
\end{align*}
\end{observation}

\begin{theorem}\label{B-P-F-D} (Borel-Pompieu type formula induced by ${}^{\psi}\mathfrak D_{a,{\bf g} }^{\vec{\alpha}}$ and ${}^{\psi}\mathfrak D_{r,a, {\bf h} }^{\vec{\beta}}$) 
Let ${\bf g} = (g_0,g_1,g_2,g_3)$ and ${\bf h}= (h_0,h_1,h_2,h_3)$ be two vector-valued functions with monotonously increasing components such that $g_k, h_k\in  C^1[a_k, b_k]$  and  $ g_k' \neq 0 \neq h_k'$ for  $k = 0,1,2,3$. If $\mathfrak{f},\texttt{f} \in AC^1(\overline{J_a^b}, \mathbb H)$,    consider $q\in J_a^b$ such that the mappings $x\mapsto {}^{\psi}\mathcal I_a^x [\mathfrak{f}](q,x, \vec{\alpha})$ and $ x\mapsto {}^{\psi} \mathcal I_a^x [\texttt{f}](q,x, \vec{\beta} )$ belong to  $ C^1(\overline{J_a^b}, \mathbb H(\mathbb C))$. Then 
 \begin{align*}   &  \  \  \ \  \int_{\partial J_a^b}  \left(  {}^{\psi}\mathfrak  K _{a,{\bf g} }^{\vec{\alpha}}  (q, x,  y ) 
\sigma_{ y }^{\psi}   {}^{\psi}\mathcal I_{a,{\bf g}}^{ y } [\mathfrak{f}](q, y , \vec{\alpha})  
  +       {}^{\psi}\mathcal I_{a,{\bf h}  }^{ y } [\texttt{f}](q, y , \vec{\beta })
 \sigma_{ y }^{\psi}        {}^{\psi}\mathfrak  K _{a,{\bf h}  }^{\vec{\beta }}  (q, x,  y ) 
\right)\\
     & -  
\int_{J_a^b} \left(  
{}^{\psi}\mathfrak  K _{a,{\bf g} }^{\vec{\alpha}}  (q, x,  y )   {}^{\psi}\mathfrak C_{a,{\bf g} }^{\vec{\alpha}}[\mathfrak{f}](q, y)  
  +    
{}^{\psi}\mathfrak C_{r,a,{\bf h} }^{\vec{\beta }}[\texttt{f}](q, y) 	{}^{\psi}\mathfrak  K _{a,{\bf h} }^{\vec{\beta }}  (q, x,  y )     
	\right  )  dy     \\
& -   
\int_{J_a^b} \left(  {}^{\psi}\mathfrak  K _{a,{\bf g} }^{\vec{\alpha}}  (q, x,  y ) {}^{\psi}\mathfrak D_{a,{\bf g} }^{\vec{\alpha}}[\mathfrak{f}](q, y)  
    +  {}^{\psi}\mathfrak D_{r,a,{\bf h} }^{\vec{\beta}}[\texttt{f}](q, y) {}^{\psi}\mathfrak  K _{a,{\bf h}}^{\vec{\beta}}  (q, x,  y )   
	\right)
	\\
		=  &      \left\{ \begin{array}{ll} 
		\displaystyle \sum_{j= 0}^3   (\mathfrak{f}+\texttt{f}) (q_0, \dots, x_k , \dots, q_3)  +   M_{a,g}^{\vec \alpha}[\mathfrak{f}](q,x) + M_{a,h}^{\vec \beta}[\texttt{f}](q,x)  
		, &     x\in 
		J_a^b ,  \\ 0 , &  x\in \mathbb H\setminus\overline{J_a^b},                   
\end{array} \right. 
\end{align*}
where 
$${}^{\psi}\mathfrak  K _{a,{\bf g} }^{\vec{\alpha}}  (q, x,  y )  = {}^{\psi}\mathfrak P_{a,{\bf g} }^{\vec{\alpha}} [ K_{\psi}( y -x) ] (q, x)$$
and  
$$M_{a,{\bf g}}^{\vec \alpha}[\mathfrak{f}](q,x) = \sum_{{   { \begin{array}{c}j, k=0 \\ j\neq k  \end{array}} }}^3 
 \frac{( {\bf I}_{a_k ^+,g_k }^{    { 1- \alpha_{k }  } } \mathfrak{f} )(q_0, \dots, x_k , \dots, q_3)}{ \Gamma(1-\alpha_j)   (   g_j(x_j) - g_j(a) )^{1-\alpha_j}}.$$
\end{theorem} 

 \begin{proof} 
Application of formula \eqref{BorelHyp} with $\mathfrak{f}$ and $\texttt{f}$ replaced by ${}^{\psi}\mathcal I_{a,{\bf g} } ^x [\mathfrak{f}](q,x,\vec{\alpha})$ and ${}^{\psi}\mathcal I_{a,{\bf h} }^x [\texttt{f}](q,x, \vec{\beta})$ enables us to write 
 \begin{align}\label{ecua1}  &  \int_{\partial J_a^b} (K_{\psi}( y -x)\sigma_{ y }^{\psi}  {}^{\psi}\mathcal I_{a,{\bf g} }^{ y } [\mathfrak{f}](q, y , \vec{\alpha}) + {}^{\psi} \mathcal I_{a,{\bf h} }^{ y } [\texttt{f}](q, y ,\vec{\beta}) \sigma_{ y }^{\psi} K_{\psi}( y -x) ) \nonumber  \\ 
&  -   
\int_{J_a^b} (K_{\psi} (y-x) {}^{\psi}\mathcal D_{y} {}^{\psi} \mathcal I_{a,{\bf g} }^{y} [\mathfrak{f}] (q,y, \vec{\alpha}) +	{}^{\psi}\mathcal D_{r,y} {}^{\psi}\mathcal I_{a,{\bf h} }^{y} [\texttt{f}](q,y, \vec{\beta}) K_{\psi} (y-x))dy   \nonumber \\
		=  &      \left\{ \begin{array}{ll} {}^{\psi} \mathcal I_{a,{\bf g} }^{x} [\mathfrak{f}](q,x,\vec{\alpha}) +  {}^{\psi}\mathcal I_{a,{\bf h} }^{x} [\texttt{f}](q,x, \vec{\beta})  , &     x\in 
		J_a^b ,  \\ 0 , &  x\in \mathbb H\setminus\overline{J_a^b}.                   
\end{array} \right. 
\end{align}  
From 1. in Proposition \ref{identities} we get 
 \begin{align*}   &  \int_{\partial J_a^b} (K_{\psi}( y -x)\sigma_{ y }^{\psi}  {}^{\psi}\mathcal I_{a,{\bf g} }^{ y } [\mathfrak{f}](q, y , \vec{\alpha}) + {}^{\psi} \mathcal I_{a,{\bf h} }^{ y } [\texttt{f}](q, y ,\vec{\beta}) \sigma_{ y }^{\psi} K_{\psi}( y -x) ) \nonumber  \\ 
&  -   
\int_{J_a^b} (K_{\psi} (y-x) {}^{\psi}\mathfrak C_{a,{\bf g} }^{\vec{\alpha}}[\mathfrak{f}](q, y)  
 +  {}^{\psi}\mathfrak C_{r,a,{\bf h} }^{\vec{\beta}}[\texttt{f}](q, y)   K_{\psi} (y-x))dy   \nonumber \\
& -   
\int_{J_a^b} (K_{\psi} (y-x)  {}^{\psi}\mathfrak D_{a,{\bf g} }^{\vec{\alpha}}[\mathfrak{f}](q, y)  
 +	  {}^{\psi}\mathfrak D_{r,a,{\bf h} }^{\vec{\beta}}[\texttt{f}](q, y)
   K_{\psi} (y-x))dy   \nonumber \\
		=  &      \left\{ \begin{array}{ll} {}^{\psi} \mathcal I_{a,{\bf g} }^{x} [\mathfrak{f}](q,x,\vec{\alpha}) + 
		{}^{\psi}\mathcal I_{a,{\bf h}}^{x} [\texttt{f}](q,x, \vec{\beta})  , &     x\in 
		J_a^b ,  \\ 0 , &  x\in \mathbb H\setminus\overline{J_a^b},                   
\end{array} \right. 
\end{align*}  
Suppose $\texttt{f}=0$, applying the operator ${}^{\psi}\mathfrak P_{a,{\bf g} }^{\vec{\alpha}}$ on both sides using Fact 3. of Proposition \ref{identities}, then Leibniz rule implies that 
 \begin{align*}   &  \int_{\partial J_a^b}   {}^{\psi}\mathfrak P_{a,{\bf g} }^{\vec{\alpha}} [ K_{\psi}( y -x) ]  
(q, x)
\sigma_{ y }^{\psi}   {}^{\psi}\mathcal I_{a,{\bf g} }^{ y } [\mathfrak{f}](q, y , \vec{\alpha})     -   
\int_{J_a^b} {}^{\psi}\mathfrak P_{a,{\bf g} }^{\vec{\alpha}} [ K_{\psi}( y -x) ]  
(q, x)   {}^{\psi}\mathfrak C_{a,{\bf g} }^{\vec{\alpha}}[\mathfrak{f}](q, y)  
  dy     \\
& -   
\int_{J_a^b} {}^{\psi}\mathfrak P_{a,{\bf g} }^{\vec{\alpha}} [ K_{\psi}( y -x) ]  
(q, x) {}^{\psi}\mathfrak D_{a,{\bf g} }^{\vec{\alpha}}[\mathfrak{f}](q, y)  
   \\
		=  &      \left\{ \begin{array}{ll} 
		\displaystyle \sum_{j= 0}^3   \mathfrak{f}  (q_0, \dots, x_k , \dots, q_3)  +  \sum_{   { \begin{array}{c}j, k=0 \\ j\neq k  \end{array}} }^3 
 \frac{( {\bf I}_{a_k ^+,g_k }^{    { 1- \alpha_{k }  } } \mathfrak{f} )(q_0, \dots, x_k , \dots, q_3)}{ \Gamma(1-\alpha_j)   (   g_j(x_j) - g_j(a) )^{1-\alpha_j}  }  
		, &     x\in 
		J_a^b ,  \\ 0 , &  x\in \mathbb H\setminus\overline{J_a^b}.                 
\end{array} \right. 
\end{align*}
Suppose now $\mathfrak{f}=0$ in \eqref{ecua1} and compute similarly as before for $\texttt{f}$. After addition of the two readily inferred relations, the theorem follows.
\end{proof}

\begin{observation} To provide an explicit representation of ${}^{\psi}\mathfrak  K _{a,{\bf g} }^{\vec{\alpha}}  (y, x, \tau)$ one can 
use a decomposition of the hyperholomorphic Cauchy kernel in terms of Gegenbauer polynomials given in \cite[page 93]{GS2} such that  
$${}^{\psi}\mathfrak  K _{a,{\bf g} }^{\vec{\alpha}}  (q, x, y) :=
		\frac{1}{2\pi^2} \sum_{k=0}^{\infty}  \frac{1}{|y|^{k+3}}	{}^{\psi}\mathfrak P_{a,g}^{\vec{\alpha}} 
		\left[  |x|^k A_{4,k} (x,y) \right] (q, x) ,$$
		with
		$$2A_{4,k} (x,y):= [ (k+1)  C_{k+1}^{1} (s) +  (2-n) C_{k}^{2} (s) \omega_y \wedge \omega_x]  \bar \omega_x, $$
where  $C_{k+1}^{1} $ and $ C_{k}^{2}$ are the Gegenbauer polynomials, $x= |x|\omega_x$, $y= |y|\omega_y$ and   $s=(\omega_x, \omega_y)$. 
\end{observation}

\begin{corollary}\label{C_I_F} (Cauchy type formula induced by ${}^{\psi}\mathfrak D_{a,{\bf g} }^{\vec{\alpha}}$ and ${}^{\psi}\mathfrak D_{r,a, {\bf h} }^{\vec{\beta}}$) Suppose that ${\bf g} = (g_0,g_1,g_2,g_3)$ and ${\bf h} = (h_0,h_1,h_2,h_3)$ are two vector-valued functions with monotonously increasing components such that $g_k, h_k\in  C^1[a_k, b_k]$  and $g_k' \neq 0 \neq h_k'$ for  $k = 0,1,2,3$.  If $\mathfrak{f},\texttt{f} \in AC^1(\overline{J_a^b}, \mathbb H)$, consider $q\in J_a^b$ such that the mappings $x\mapsto {}^{\psi}\mathcal I_a^x [\mathfrak{f}](q,x, \vec{\alpha})$, \ $ x\mapsto {}^{\psi} \mathcal I_a^x [\texttt{f}](q,x, \vec{\beta} )$ belong to $ C^1(\overline{J_a^b}, \mathbb H(\mathbb C))$ and 
$${}^{\psi}\mathfrak D_{a,{\bf g} }^{\vec{\alpha}}[\mathfrak{f}](q, y) = {}^{\psi}\mathfrak D_{r,a,{\bf h} }^{\vec{\beta}}[\texttt{f}](q, y)=0.$$
Then 
 \begin{align*}   &  \int_{\partial J_a^b}  \left(  {}^{\psi}\mathfrak  K _{a,{\bf g} }^{\vec{\alpha}}  (q, x,  y ) 
\sigma_{ y }^{\psi}   {}^{\psi}\mathcal I_{a,{\bf g} }^{ y } [\mathfrak{f}](q, y , \vec{\alpha})  
  +       {}^{\psi}\mathcal I_{a,{\bf h}  }^{ y } [\texttt{f}](q, y , \vec{\beta })
 \sigma_{ y }^{\psi}        {}^{\psi}\mathfrak  K _{a,{\bf h}  }^{\vec{\beta }}  (q, x,  y ) 
\right)\\
   -   &  
\int_{J_a^b} \left(  
{}^{\psi}\mathfrak  K _{a,{\bf g} }^{\vec{\alpha}}  (q, x,  y )   {}^{\psi}\mathfrak C_{a,{\bf g} }^{\vec{\alpha}}[\mathfrak{f}](q, y)  
  +    
{}^{\psi}\mathfrak C_{r, a, { \bf h } }^{\vec{\beta }}[\texttt{f}](q, y) 	{}^{\psi}\mathfrak  K _{a,{\bf h}  }^{\vec{\beta }}  (q, x,  y )     
	\right  )  dy     \\
		=  &      \left\{ \begin{array}{ll} 
		\displaystyle \sum_{j= 0}^3   (\mathfrak{f}+\texttt{f}) (q_0, \dots, x_k , \dots, q_3)  +   M_{a,{\bf g} }^{\vec \alpha}[
		\mathfrak{f}](q,x) + M_{a,{\bf h} }^{\vec \beta}[\texttt{f}](q,x)  
		, &     x\in 
		J_a^b ,  \\ 0 , &  x\in \mathbb H\setminus\overline{J_a^b},                   
\end{array} \right. 
\end{align*}
\end{corollary}

\begin{observation}
If $g_k$ and $h_k$ are the identities functions for $k=0,1,2,3$ then the previous Borel-Pompieu and Cauchy type formulas reduce to that appear in \cite{GB}. Simultaneously, if  $g_k = h_k =\ln$ for $k=0,1,2,3$ they induce analogous formulas associated to the Hadamard fractional derivative.  
\end{observation} 
 
\begin{observation}
As an particular case, under the hypothesis of Theorem \ref{B-P-F-D} with $g_j\in C^2(a_j,b_j)$ for all $j$, and appealing to Theorem \ref{BorelHyp} for $$(g_j'(y_j) )^{-1} ({\bf I}_{a_j^+,g_j}^{    { 1- \alpha_{j}  } } \mathfrak{f} )(q_0, \dots, y_j, \dots, q_3)$$ for $j=0,1,2,3$ yields
 \begin{align*}   &  \int_{\partial J_a^b} K_{\psi}( y -x)\sigma_{ y }^{\psi}   (g_j'(y_j))^{-1} ({\bf I}_{a_j^+,g_j}^{    { 1- \alpha_{j}  } } \mathfrak{f} )(q_0, \dots, y_j, \dots, q_3)  \\
 & -   
\int_{J_a^b} K_{\psi} (y-x) {}^{\psi}\mathcal D_{y}  [(g_j'(y_j) )^{-1}  ({\bf I}_{a_j^+,g_j}^{    { 1- \alpha_{j}  } } \mathfrak{f} )(q_0, \dots, y_j, \dots, q_3)] dy   \nonumber \\
		=  &      \left\{ \begin{array}{ll}   (g_j'(x_j))^{-1}  ({\bf I}_{a_j^+,g_j}^{    { 1- \alpha_{j}  } } \mathfrak{f} )(q_0, \dots, x_j, \dots, q_3)  , &     x\in 
		J_a^b ,  \\ 0 , &  x\in \mathbb H\setminus\overline{J_a^b}.                   
\end{array} \right. 
\end{align*}  
Therefore
\begin{align*}   &  \int_{\partial J_a^b} K_{\psi}( y -x)\sigma_{ y }^{\psi}   (g_j'(y_j))^{-1} ({\bf I}_{a_j^+,g_j}^{    { 1- \alpha_{j}  } } \mathfrak{f} )(q_0, \dots, y_j, \dots, q_3)  \\
 & -   
\int_{J_a^b} K_{\psi} (y-x)  {\psi}_j \frac{1}{g_j'(y_j) }  \frac{\partial}{\partial y_j}({\bf I}_{a_j^+,g_j}^{    { 1- \alpha_{j}  } } \mathfrak{f} )(q_0, \dots, y_j, \dots, q_3)] dy   \nonumber \\
 & +
\int_{J_a^b} K_{\psi} (y-x)  {\psi}_j \frac{g_j''(y_j)}{(g_j'(y_j))^2 }  ({\bf I}_{a_j^+,g_j}^{    { 1- \alpha_{j}  } } \mathfrak{f} )(q_0, \dots, y_j, \dots, q_3)] dy   \nonumber
\\
		=  &      \left\{ \begin{array}{ll}   (g_j')^{-1}(x_j)  ({\bf I}_{a_j^+,g_j}^{    { 1- \alpha_{j}  } } \mathfrak{f} )(q_0, \dots, x_j, \dots, q_3)  , &     x\in 
		J_a^b ,  \\ 0 , &  x\in \mathbb H\setminus\overline{J_a^b}                    
\end{array} \right. 
\end{align*} and adding all terms we get the following:
\begin{align*}   &  \int_{\partial J_a^b} K_{\psi}( y -x)\sigma_{ y }^{\psi}  \sum_{j=0}^3 (g_j'(y_j))^{-1} ({\bf I}_{a_j^+,g_j}^{    { 1- \alpha_{j}  } } \mathfrak{f} )(q_0, \dots, y_j, \dots, q_3)  \\
 & -   
\int_{J_a^b} K_{\psi} (y-x)  {}^{\psi}\mathfrak D_{a,{\bf g} }^{\vec{\alpha}}[\mathfrak{f}] (q,y) dy   \nonumber \\
 & +
\int_{J_a^b} K_{\psi} (y-x)  \sum_{j=0}^3   {\psi}_j \frac{g_j''(y_j)}{(g_j'(y_j))^2 }  ({\bf I}_{a_j^+,g_j}^{    { 1- \alpha_{j}  } } \mathfrak{f} )(q_0, \dots, y_j, \dots, q_3)] dy   \nonumber
\\
		=  &      \left\{ \begin{array}{ll}  \displaystyle  \sum_{j=0}^3  (g_j'(x_j) )^{-1} ({\bf I}_{a_j^+,g_j}^{    { 1- \alpha_{j}  } } \mathfrak{f} )(q_0, \dots, x_j, \dots, q_3)  , &     x\in 
		J_a^b ,  \\ 0 , &  x\in \mathbb H\setminus\overline{J_a^b}.                   
\end{array} \right. 
\end{align*}
If ${}^{\psi}\mathfrak D_{a,{\bf g} }^{\vec{\alpha}}[\mathfrak{f}] (q,\cdot)\equiv 0$ on $J_a^b$ we have that 

\begin{align*}   &  \int_{\partial J_a^b} K_{\psi}( y -x)\sigma_{ y }^{\psi}  \sum_{j=0}^3 (g_j'(y_j) )^{-1}({\bf I}_{a_j^+,g_j}^{    { 1- \alpha_{j}  } } \mathfrak{f})(q_0, \dots, y_j, \dots, q_3)  \\
 & +
\int_{J_a^b} K_{\psi} (y-x)  \sum_{j=0}^3   {\psi}_j \frac{g_j''(y_j)}{(g_j'(y_j))^2 }  ({\bf I}_{a_j^+,g_j}^{    { 1- \alpha_{j}  } } \mathfrak{f} )(q_0, \dots, y_j, \dots, q_3)] dy   \nonumber
\\
		=  &      \left\{ \begin{array}{ll}  \displaystyle  \sum_{j=0}^3  (g_j'(x_j) )^{-1} ({\bf I}_{a_j^+,g_j}^{    { 1- \alpha_{j}  } } \mathfrak{f} )(q_0, \dots, x_j, \dots, q_3)  , &     x\in 
		J_a^b ,  \\ 0 , &  x\in \mathbb H\setminus\overline{J_a^b}.                  
\end{array} \right. 
\end{align*}
\end{observation}

\section*{Acknowledgments}
The authors wish to thank the Instituto Polit\'ecnico Nacional (grant numbers SIP20220017, SIP20221274)} for partial support.

\end{document}